\documentclass{amsart}

\usepackage{amssymb,colonequals,enumerate}
\usepackage[curve,matrix,arrow]{xy}

\usepackage{pdfsync} \usepackage[colorlinks]{hyperref}
\usepackage{mathtools} \usepackage{hyperref} 

\theoremstyle{plain}

\newtheorem{theorem}[section]{Theorem}
\newtheorem{proposition}[section]{Proposition}

\newtheorem{corollary}[section]{Corollary}

\theoremstyle{definition}

\theoremstyle{remark}

\newtheorem{remark}[section]{Remark}

\def\urltilda{\kern -.15em\lower .7ex\hbox{\~{}}\kern .04em}
\newcommand{\codim}{\operatorname{codim}}
\renewcommand{\dim}{\operatorname{dim}}
\newcommand{\depth}{\operatorname{depth}}
\newcommand{\hh}{\operatorname{H}}
\newcommand{\Hom}{\operatorname{Hom}}
\newcommand{\rank}{\operatorname{rank}} \newcommand{\vf}{\varphi}
\renewcommand{\le}{\leqslant}
\renewcommand{\ge}{\geqslant}
\newcommand{\Ext}{\operatorname{Ext}}
\newcommand{\Spec}{\operatorname{Spec}}
\newcommand{\supp}{\operatorname{supp}}
\newcommand{\Supp}{\operatorname{Supp}}
\newcommand{\projdim}{\operatorname{proj.dim}}
\newcommand{\Ltensor}{\otimes^{\mathsf{L}}}
\newcommand{\lra}{\longrightarrow}
\newcommand{\bbZ}{\mathbb Z}
\newcommand{\fm}{\mathfrak{m}} \newcommand{\fp}{\mathfrak{p}}

\newcommand{\xra}[2][]{%
  \mathchoice%
  {\xrightarrow[#1]{\hspace{.58em}\mathclap{#2}\hspace{.58em}}}
  {\xrightarrow[#1]{\hspace{.44em plus
        1em}\mathclap{#2}\hspace{.44em}}} { \ }{ \ } }

\begin{document}

\title[Dimension of finite free complexes]{Dimension of finite free
  complexes over\\ commutative Noetherian rings}

\author[L.\,W.\ Christensen]{Lars Winther Christensen} \address{Texas
  Tech University, Lubbock, TX 79409, U.S.A.}
\email{lars.w.christensen@ttu.edu}
\urladdr{http://www.math.ttu.edu/\urltilda lchriste}

\author[S.\,B.~Iyengar]{Srikanth B. Iyengar} \address{University of
  Utah, Salt Lake City, UT 68588, U.S.A.}
\email{iyengar@math.utah.edu}
\urladdr{http://www.math.utah.edu/\urltilda iyengar}

\thanks{L.W.C.\ was partly supported by Simons Foundation
  collaboration grant 428308; S.B.I.\ was partly supported by NSF
  grant DMS-1700985.}

\date{9 September 2020}



\dedicatory{To Roger and Sylvia Wiegand on the occasion of their
  aggregate $150^{th}$ birthday.}

\begin{abstract}
  Foxby defined the (Krull) dimension of a complex of modules over a
  commutative Noetherian ring in terms of the dimension of its
  homology modules. In this note it is proved that the dimension of a
  bounded complex of free modules of finite rank can be computed
  directly from the matrices representing the differentials of the
  complex.
\end{abstract}

\maketitle

\section*{Introduction} 

\noindent
This short note concerns certain homological
invariants---specifically, dimension and depth---of complexes of
modules over commutative Noetherian local rings. The concepts of depth
and dimension for modules, introduced by Krull and by Auslander and
Buchsbaum, respectively, need no recollection. Both concepts were
extended to complexes of modules by Foxby \cite{HBF79}, and also by
Iversen \cite{BIv77}. Their extensions agree up to a normalization; in
what follows we work with Foxby's definitions, recalled further below,
for they are better suited to computations in the derived category.
The depth and dimension of a complex depend only on the
quasi-isomorphism class of the complex; said differently, they are
defined on the derived category of the ring.

To compute these invariants one can usually reduce to the case where
the complex is finite free, for they are independent of the
domain. Indeed, if $Q\to R$ is a surjective map of rings with $Q$ a
regular local ring, then the depth and dimension of an $R$-complex $M$
coincide with the corresponding invariants of $M$ viewed as a complex
over $Q$. And, at least when $M$ is homologically finite, it is
quasi-isomorphic, over $Q$, to a finite free complex. Thus, in what
follows we consider a complex over a local ring $R$ of the form:
\begin{equation*}
  F \:\colonequals\: 0 \lra F_b\xra{\partial} \cdots \xra{\partial} F_a\lra 0
\end{equation*}
where each $F_i$ is a free $R$-module of finite rank. We assume that $F$ is
minimal, in that $\partial(F)\subseteq \fm F$, where $\fm$ is the
maximal ideal of $R$.

For such a complex $F$, the depth can be read off easily: The equality
of Auslander and Buchsbaum for modules of finite projective dimension
applies equally to complexes---this was proved by Foxby
\cite{HBF80}---and yields that the depth of $F$ equals $\depth R - b$,
provided that $F_b \ne 0$.  In this note we establish a
formula that expresses the dimension of $F$ in terms of the ranks of
the modules $F_i$ and the Fitting ideals of the differentials; see
Theorem~\ref{mainthm}. We were lead to it in an attempt to relate the
codimension, in the sense of Bruns and Herzog~\cite{bruher}, to other
homological invariants. It turns out that the codimension of $F$
equals $\dim R - \dim_R \Hom_R(F,R)$; see Remark~\ref{bhcodim}. This
observation gives a different perspective on, and different proofs of,
certain results in \cite{bruher} related to the homological
conjectures; see Proposition~\ref{sup} and Theorem~\ref{Fstar}.

\section*{Dimension} 

\noindent
Let $R$ be a ring. By an $R$-complex we mean a complex of $R$-modules,
with lower grading:
\begin{equation*}
  X \: \colonequals \: \cdots \lra X_n \xra{\partial_{n}} X_{n-1}\lra \cdots
\end{equation*}
A graded $R$-module, such as the homology $\hh(X)$ of $X$, is viewed
as an $R$-complex with zero differentials; in particular, we use the
same grading convention for such objects.

\subsection*{Dimension}
Let $R$ be a commutative Noetherian ring. In \cite[Section 3]{HBF79}
Foxby introduced the \emph{dimension} of an $R$-complex $X$ to be
\begin{equation}
  \label{dim0}
  \dim_R X \colonequals \sup\{\dim(R/\fp) -
  \inf\hh(X)_\fp \mid \fp\in\Spec R \}\:.
\end{equation}
By \cite[Proposition 3.5]{HBF79} this invariant can be computed in
terms of the homology:
\begin{equation}
  \label{dim1}
  \dim_R X = \sup\{\dim_R \hh_n(X) - n \mid n\in \bbZ\}\:.
\end{equation}
The convention is that the dimension of the zero module is $-\infty$.

\subsection*{Finite free complexes}
By a \emph{finite free} $R$-complex we mean a bounded $R$-complex
\begin{equation}
  \label{ffc}
  F \:\colonequals\: 0\lra F_b\xra{\partial_b} F_{b-1} \lra \cdots \lra
  F_{a+1} \xra{\partial_{a+1}}F_a\lra 0
\end{equation}
where each $F_i$ is a free $R$-module of finite rank. For such a
complex $F$ we set
\begin{equation}
  \label{sn}
  s_n = \sum_{i \le n}(-1)^{n-i} \rank_RF_i \quad\text{for each $n\in\bbZ$}\:.
\end{equation}
Given a map $\vf$ between finite free modules we write $I_s(\vf)$ for
the ideal generated by the $s\times s$ minors of a matrix representing
$\vf$; see, for example, \cite[p.~21]{bruher}. It is convenient to
adopt the convention that the determinant and all minors of the empty
matrix is $1$; in particular, for $s \le 0$ an $s \times s$ minor of
any matrix is $1$. For $s \ge 1$ an $s\times s$ minor of a non-empty
matrix is $0$ if $s$ exceeds the number of rows or columns. For the
differentials of the complex \eqref{ffc} this means that
$I_{s_{n}}(\partial_{n+1}) = R$ holds for integers $n$ outside
$[a,b]$.

\begin{theorem}
  \label{mainthm}
  With $F$ and $s_n$ as above, there is an equality
  \begin{equation*}
    \dim_R F = \sup\{\dim (R/I_{s_{n}}(\partial_{n+1})) - n
    \mid n\in\bbZ \}\:.
  \end{equation*}
\end{theorem}

\begin{proof}
  To prove the inequality ``$\ge$'' we verify that
  \begin{equation*}
    \dim (R/I_{s_{n}}(\partial_{n+1})) \le \dim_R F + n
  \end{equation*}
  holds for every integer $n$ in $[a,b]$.  Fix such an $n$. The
  inequality above holds if and only if one has
  $I_{s_{n}}(\partial_{n+1})_\fp=R_\fp$ for every $\fp\in\Spec R$ with
  $\dim(R/\fp) > \dim_R F + n$.  For such a prime ideal and any
  integer $i\le n$ one gets
  \begin{equation*}
    \dim_R\hh_i(F)\le \dim_R F + i <\dim(R/\fp)
  \end{equation*}
  where the first inequality holds by \eqref{dim1}.  This yields
  \begin{equation*}
    \hh_i(F)_\fp =0 \quad \text{for all $i\le n$ }\:,
  \end{equation*}
  which implies that the homology of the complex
  \begin{equation}
    \label{eq:Fp}
    (F_{n+1})_\fp \xrightarrow{(\partial_{n+1})_\fp} (F_{n})_{\fp} \lra
    \cdots \lra  (F_{a+1})_\fp \xrightarrow{(\partial_{a+1})_\fp}(F_{a})_{\fp} \lra 0
  \end{equation}
  is zero in degrees $\le n$. It follows that the image of
  $(\partial_{n+1})_{\fp}$ is a free $R_{\fp}$-module of rank
  $s_n$. Hence one has $I_{s_{n}}(\partial_{n+1})_\fp=R_\fp$.
  
  To prove the opposite inequality, ``$\le$'', we show that
  \begin{equation*}
    \dim_R \hh_{n}(F) - n \leq \sup\{\dim (R/I_{s_i}(\partial_{i+1})) - i
    \mid a \le i \le b\}
  \end{equation*}
  holds for each integer $n$ in $[a,b]$. Let $t$ be the
  supremum above. One needs to verify that $\hh_{n}(F)_{\fp}=0$ holds
  for primes $\fp$ with $\dim(R/\fp) > t + n$. Fix such a $\fp$; for
  every $i\le n$ one has
  \begin{equation*}
    \dim (R/I_{s_i}(\partial_{i+1})) \le t+ i \le t+n <\dim(R/\fp)
  \end{equation*}
  so that $I_{s_i}(\partial_{i+1})_{\fp} = R_{\fp}$. We now argue by
  induction on $i$ that the homology of the complex \eqref{eq:Fp} is
  zero in degrees $\le n$; in particular, one has
  $\hh_{n}(F)_{\fp}=0$, as desired. With $f_i = \rank_RF_i$ and
  $K_i = \ker \partial_i$ the argument goes as follows: In the base
  case $i=a$ one applies \cite[Lemma~1.4.9]{bruher} to the presentation of
  the image of $\partial_{a+1}$ afforded by \eqref{eq:Fp}, and one
  concludes that it is a free submodule of $F_a$ of rank $f_a$, i.e.\
  the whole thing. One also notices that a free module contained in
  $K_{a+1}$ has rank at most $s_{a+1}$. In the induction step one
  applies \emph{op.cit} to the presentation of the image of
  $\partial_{i+1}$ and concludes that it is a free module of rank
  $f_{i+1} - s_{i+1}=s_i$. By the induction hypothesis a free module
  contained in $K_i$ has rank at most $s_i$, so the complex is exact
  at $(F_i)_\fp$.
\end{proof}

\subsection*{Codimension}
Let $R$ be a commutative Noetherian ring and $F$ a finite free
$R$-complex as in \eqref{ffc}. For each integer $n$ set
\begin{equation*}
  r_n\colonequals \sum_{i\ge n} (-1)^{i-n} \rank_R(F_i)\:.
\end{equation*}
For $n$ in $[a+1,b]$ this is the \emph{expected rank} of the map
$\partial_n$; see \cite[p.~24]{bruher}.

\begin{corollary}
  \label{codim1}
  With $F$ and $r_n$ as above there is an equality
  \begin{equation*}
    \dim_R \Hom_R(F,R) = \sup\{\dim(R/I_{r_n}(\partial^F_n)) + n
    \mid n \in \bbZ \}\:.
  \end{equation*}
\end{corollary}

\begin{proof}
  Set $G\colonequals \Hom_R(F,R)$. This too is a finite free complex,
  concentrated in degrees $[-b,-a]$, with differentials
  $\partial^{G}_{n} = \Hom_R(\partial^F_{1-n},R)$ for each $n$.  It is
  now easy to check that the expected ranks $r_n$ of $F$ and the
  invariants $s_n$ of $G$, from \eqref{sn}, determine each other:
  \begin{equation*}
    s_n(G) = r_{-n}(F)\quad\text{for each $n$}.
  \end{equation*}
  Whence one gets equalities
  \begin{align*}
    \dim_R G 
    &= \sup\{\dim (R/I_{s_n}(\partial^G_{n+1})) - n\mid n\in \bbZ\} \\
    &=  \sup\{\dim (R/I_{r_{-n}}(\partial^F_{-n}) -n \mid n\in\bbZ\} \\
    &=  \sup\{\dim (R/I_{r_n}(\partial^F_n)) + n  \mid n \in \bbZ \}\:.
      \qedhere
  \end{align*}
\end{proof}

\begin{remark}
  \label{bhcodim}
  Bruns and Herzog~\cite[Section 9.1]{bruher} have introduced a notion
  of ``codimension" for finite free complexes. This is perhaps a
  misnomer: Applied to the minimal free resolution of a module, the
  codimension does not equal the usual codimension of the module.  In
  fact, Corollary~\ref{codim1} yields that the codimension, in their
  sense, of any finite free $R$-complex $F$, is precisely
  $\dim R - \dim_R \Hom_R(F,R)$.

  Foxby also has a notion of codimension for an $R$-complex $X$,
  namely the invariant
  \begin{align*}
    \codim_R X & \colonequals \inf\{\dim R_\fp +
                 \inf\hh(X)_\fp \mid \fp\in\Spec R \} \\
               & =     \inf\{\codim_R \hh_n(X) + n \mid n\in \bbZ\}\:;
  \end{align*}
  see \cite[Lemma 5.1]{HBF79} and the definition preceding it. From
  the definitions one immediately gets
  $\codim_R X + \dim_R X \le \dim R$; equality holds if $R$ is local,
  catenary, and equidimensional. For a finite free complex $F$ over
  such a ring one thus has
  \begin{equation*}
    \codim_R \Hom_R(F,R) = \dim R - \dim_R\Hom_R(F,R)\:.
  \end{equation*}
  In particular, the codimension of $F$ in the sense of \cite{bruher}
  is the codimension of the dual complex, $\Hom_R(F,R)$, in the sense
  of \cite{HBF79}.
\end{remark}

In Bruns and Herzog's \cite{bruher} treatment of the homological
conjectures---most of which are now theorems thanks to
Andr\'e~\cite{YAn18}---their notion of codimension of a finite free
complex is key. Per Remark~\ref{bhcodim} this suggests that estimates
on the dimension of $\Hom_R(F,R)$ are useful, and that motivates the
development below.

\subsection*{Support} Let $R$ be a commutative Noetherian ring. The
\emph{large support} of an $R$-complex $X$ is the support of the
graded module $\hh(X)$, i.e.
\begin{equation*}
  \Supp_R X \colonequals \{\fp\in\Spec R \mid \hh_n(X)_\fp \ne 0
  \text{ for some $n$}\}\:.
\end{equation*}
Foxby \cite[Section 2]{HBF79} also introduced the (small)
\emph{support} of $X$ to be the set
\begin{equation*}
  \supp_R X \colonequals \{\fp\in\Spec R
  \mid \hh(\kappa(\fp) \Ltensor_R X) \ne 0 \}\:;
\end{equation*}
as usual, $\kappa(\fp)$ denotes the residue field of the local ring
$R_\fp$.  Support is connected to the finiteness of the depth of $X$:
\begin{equation*}
  \label{supp1}
  \supp_R X = \{\fp\in \Spec R \mid \depth_{R_\fp} X_\fp  < \infty \}\:.
\end{equation*}

We recall that the \emph{depth} of a complex $X$ over local ring $R$
with residue field $k$ is
\begin{equation*}
  \depth_RX \colonequals \inf\{n \in \bbZ \mid \Ext^n_R(k,X)\ne 0\}\:.
\end{equation*}

This invariant can also be computed in terms of the Koszul homology,
and the local cohomology, of $X$; see \cite{HBFSIn03}.

\begin{proposition}
  \label{fstar}
  Let $R$ be a commutative noetherian ring. For every finite free
  $R$-complex $F$ one has
  \begin{equation*}
    \dim_R \Hom_R(F,R) \le \dim R + \sup \hh(F) \:.
  \end{equation*}
\end{proposition}

\begin{proof}
  For every prime ideal $\fp$ the complex $F_\fp$ has finite
  projective dimension, so the Auslander--Buchsbaum Formula combines
  with standard (in)equalities between invariants to yield
  \begin{align*}
    -\inf\hh(\Hom_R(F,R))_\fp 
    & = \sup\{m\in\bbZ \mid \Ext_{R_\fp}^m(F_\fp,R_\fp) \ne 0\} \\
    & = \projdim_{R_\fp}F_\fp \\
    & = \depth R_\fp - \depth_{R_\fp}F_\fp \\
    & \le \dim R_\fp + \sup \hh(F)_\fp \:.
  \end{align*}
  From the definition \eqref{dim0} one now gets
  \begin{align*}
    \dim_R\Hom_R(F,R) &\le \sup\{\dim (R/\fp) + \dim R_\fp +
                        \sup \hh(F)_\fp \mid \fp\in\Spec R \}  \\
                      &\le \dim R + \sup \hh(F)\:. \qedhere
  \end{align*}
\end{proof}

\begin{remark}
  For the minimal free resolution of a finitely generated module of
  finite projective dimension, the codimension considered in
  \cite{bruher} is non-negative by the Buchsbaum--Eisenbud acyclicity
  criterion; see the comment before \cite[Lemma 9.1.8]{bruher}. This
  compares to the inequality in Proposition~\ref{fstar}, rewritten
  as
  \begin{equation*}
    \dim R - \dim_R\Hom_R(F,R) \ge -\sup\hh(F) \:.
  \end{equation*}
\end{remark}

\subsection*{Balanced big Cohen--Macaulay modules}
Let $(R,\fm)$ be local and $M$ a big Cohen--Macaulay module; that is,
a module with $\depth_RM=\dim R$ and $\fm M\ne
M$. Hochster~\cite{MHc73a,MHc75} proved that such a module exists for
every equicharacteristic local ring, and Andr\'e \cite{YAn18a} proved
their existence over local rings of mixed characteristic. A big
Cohen--Macaulay $R$-module $M$ is called \emph{balanced} if every
system of parameters for $R$ is an $M$-regular sequence. The
$\fm$-adic completion of any big Cohen--Macaulay module is balanced;
see \cite[Theorem 8.5.3]{bruher}. Sharp~\cite{RYS81} demonstrated that
these modules behave much like maximal Cohen--Macaulay modules. Of
interest here is the fact that for a balanced big Cohen--Macaulay
module $M$ one has
\begin{equation}
  \label{bbcm}
  \depth_{R_\fp} M_\fp = \dim R - \dim (R/\fp) \quad 
  \text{for each $\fp \in\supp_RM$\:;}
\end{equation}
this is part $(iii)$ in \cite[Theorem 3.2]{RYS81}.
Note that what Sharp calls the supersupport of $M$ is the support of
$M$, in the sense above; this follows from comparison of
\cite[Remark~2.9]{HBF79} and part $(v)$ in \emph{op.\:cit.}

\begin{proposition}
  \label{sup}
  Let $R$ be a local ring, $F$ a finite free $R$-complex, and $M$ a
  balanced big Cohen--Macaulay module. One has
  \begin{equation*}
    \sup\hh(F \otimes_R M) =  \dim_R \Hom_R(F,R) - \dim R \:.
  \end{equation*}
\end{proposition}

\begin{proof}
  Set $G \colonequals \Hom_R(F,R)$. There is an isomorphism
  $F \otimes M \cong \Hom_R(G,M)$. In the computation below, the first
  equality holds by \cite[Proposition 3.4]{HBF79}. The second equality
  follows from \eqref{bbcm} and the fourth one follows from
  \eqref{dim0}.
  \begin{align*}
    \sup\hh(\Hom_R(G,M)) 
    & = -\inf\{\depth_{R_\fp} M_\fp + \inf \hh(G)_\fp \mid \fp\in\Spec R \} \\
    & = -\inf\{\dim R - \dim (R/\fp) + \inf \hh(G)_\fp \mid \fp\in\Spec R \} \\
    & = \sup\{\dim (R/\fp) - \inf \hh(G)_\fp \mid \fp\in\Spec R \} - \dim R \\
    & = \dim_R G - \dim R\:. \qedhere
  \end{align*}
\end{proof}

Given Remark~\ref{bhcodim}, the theorem above recovers \cite[Lemma
9.1.8]{bruher}:

\begin{corollary}
  \label{BH}
  Let $R$ be a local ring and
  $F \colonequals 0 \to F_b \to \cdots \to F_0 \to 0$ a finite free
  $R$-complex.  If $\dim R - \dim_R\Hom_R(F,R) \ge 0$ holds, then for
  any balanced big Cohen--Macaulay module $M$ one has
  $\hh_i(F \otimes_R M)=0$ for all $i\ge 1$. \qedhere
\end{corollary}

In \cite[Section 9.4]{bruher} it is shown how to derive various
intersections theorems, including the New Intersection Theorem, from
Corollary \ref{BH}. The latter sheds further light on the invariant
$\dim R - \dim_R\Hom_R(F,R)$.

\begin{theorem}
  \label{Fstar}
  Let $R$ be a local ring and $F$ a finite free $R$-complex. One has
  \begin{equation*}
    \dim R + \inf \hh(F) \le \dim_R \Hom_R(F,R) \le
    \dim R + \sup \hh(F)\:.
  \end{equation*}
\end{theorem}

\begin{proof}
  The right-hand inequality holds by Proposition~\ref{fstar}.  Since
  $R$ is local, one can apply the version of the New Intersection
  Theorem recorded by Foxby \cite[Lemma 4.1]{HBF79} to the complex
  $\Hom_R(F,R)$ to get
  \begin{equation*}
    \dim R - \dim_R\Hom_R(F,R) \le \projdim_R \Hom_R(F,R)\;.
  \end{equation*}
  As $\Hom_R(F,R)$ is also a finite free complex one has
  \begin{equation*}
    \projdim_R \Hom_R(F,R) = -\inf\hh(\Hom_R(\Hom_R(F,R),R)) = -\inf
    \hh(F)\:. \qedhere
  \end{equation*}
\end{proof}

\begin{remark}
  Let $R$ be a local ring and $M$ a nonzero finitely generated
  $R$-module of finite projective dimension. Applying
  Theorem~\ref{Fstar} to a finite free resolution of $M$ yields
  the equality
  \begin{equation*}
    \max\{\dim_R \Ext_R^n(M,R) + n \mid n\in \bbZ\} = \dim R\:.
  \end{equation*}
  Notice that with $p \colonequals \projdim_R M$ one gets inequalities
  \begin{equation*}
    \dim R - \dim_R M \le \projdim_R M \le \dim R - \dim_R \Ext^p_R(M,R) \:;
  \end{equation*}
  the inequality on the left is the version of the New Intersection
  Theorem that went into the proof of Theorem \ref{Fstar}.
\end{remark}


\def\soft#1{\leavevmode\setbox0=\hbox{h}\dimen7=\ht0\advance \dimen7
  by-1ex\relax\if t#1\relax\rlap{\raise.6\dimen7
  \hbox{\kern.3ex\char'47}}#1\relax\else\if T#1\relax
  \rlap{\raise.5\dimen7\hbox{\kern1.3ex\char'47}}#1\relax \else\if
  d#1\relax\rlap{\raise.5\dimen7\hbox{\kern.9ex \char'47}}#1\relax\else\if
  D#1\relax\rlap{\raise.5\dimen7 \hbox{\kern1.4ex\char'47}}#1\relax\else\if
  l#1\relax \rlap{\raise.5\dimen7\hbox{\kern.4ex\char'47}}#1\relax \else\if
  L#1\relax\rlap{\raise.5\dimen7\hbox{\kern.7ex
  \char'47}}#1\relax\else\message{accent \string\soft \space #1 not
  defined!}#1\relax\fi\fi\fi\fi\fi\fi}
  \providecommand{\MR}[1]{\mbox{\href{http://www.ams.org/mathscinet-getitem?mr=#1}{#1}}}
  \renewcommand{\MR}[1]{\mbox{\href{http://www.ams.org/mathscinet-getitem?mr=#1}{#1}}}
  \providecommand{\arxiv}[2][AC]{\mbox{\href{http://arxiv.org/abs/#2}{\sf
  arXiv:#2 [math.#1]}}} \def\cprime{$'$}
\providecommand{\bysame}{\leavevmode\hbox to3em{\hrulefill}\thinspace}
\providecommand{\MR}{\relax\ifhmode\unskip\space\fi MR }
\providecommand{\MRhref}[2]{%
  \href{http://www.ams.org/mathscinet-getitem?mr=#1}{#2}
}
\providecommand{\href}[2]{#2}

\end{document}